\documentclass[12pt]{amsart}
\usepackage{amsfonts}
\usepackage{graphicx}
\usepackage{tabularx}
\usepackage{array}
\usepackage[usenames,dvipsnames]{color}
\usepackage{comment}
\usepackage{amsmath}
\usepackage{amsthm}
\usepackage{amssymb}
\usepackage{fullpage}
\usepackage[dvipsnames]{xcolor}
\usepackage{listings}

\DeclareMathOperator{\inj}{inj}
\DeclareMathOperator{\dist}{dist}

\newcommand{\chinj}{\chi'_{\inj}}

\newtheorem{theorem}{Theorem}[section]
\newtheorem{proposition}[theorem]{Proposition}

\newtheorem{conjecture}[theorem]{Conjecture}
\newtheorem{question}[theorem]{Question}

\newtheorem{lemma}[theorem]{Lemma}

\theoremstyle{definition}

\author{Peter Bradshaw}
\address{Department of Mathematics, University of Illinois Urbana-Champaign}
\email{pb38@illinois.edu}
\thanks{This project was partially funded by NSF RTG grant DMS-1937241, NSF grant DMS-2153507, and NSERC through the Canadian Graduate Scholarship - Master program}

\author{Alexander Clow}
\address{Department of Mathematics, Simon Fraser University}
\email{alexander\_clow@sfu.ca}

\author{Jingwei Xu}
\address{Department of Mathematics, University of Illinois Urbana-Champaign}
\email{jx6@illinois.edu}

\def\epsilon{\varepsilon}

\title{Injective edge colorings of degenerate graphs \\ and the oriented chromatic number}
\begin{document}
\maketitle
\begin{abstract}
Given a graph $G$, an \emph{injective edge-coloring} of $G$ is a function $\psi:E(G) \rightarrow \mathbb N$ such that if $\psi(e) = \psi(e')$, then no 
third edge joins an endpoint of $e$
and an endpoint of $e'$.
The \emph{injective chromatic index} of a graph $G$, written $\chinj(G)$, is the minimum number of colors needed for an injective edge coloring of $G$. In this paper, we investigate the injective chromatic index of certain classes of degenerate graphs. First, we show that if $G$ is a $d$-degenerate graph of maximum degree $\Delta$, then $\chinj(G) = O(d^3 \log \Delta)$. Next, we show that if $G$ is a graph of Euler genus $g$, then $\chinj(G) \leq (3+o(1))g$, which is tight when $G$ is a clique.
Finally, we show
that the oriented chromatic number of a graph is at most exponential in its injective chromatic index. Using this fact, we prove that the oriented chromatic number of a graph embedded on a surface of Euler genus $g$ has oriented chromatic number at most $O(g^{6400})$, improving the previously known upper bound of $2^{O(g^{\frac{1}{2} + \epsilon})}$ and resolving a conjecture of Aravind and Subramanian.
\end{abstract}
\section{Introduction}

\subsection{Background: Injective chromatic index}
Given a graph $G$, an \emph{injective edge-coloring} of $G$ is a function $\psi:E(G) \rightarrow \mathbb N$ such that if $\psi(e) = \psi(e')$ for distinct edges $e,e' \in E(G)$, then no
third edge of $G$ joins an endpoint of $e$ to an endpoint of $e'$. In other words, if $\psi(e) = \psi(e')$, then $e$ and $e'$ are not at distance $1$ and do not belong to a common triangle in $G$.
The \emph{injective chromatic index} of $G$ is the minimum integer $k$ for which $G$ has an injective edge coloring $\phi:E(G) \rightarrow \{1,\dots,k\}$.
Note that an edge coloring $\phi:E(G) \rightarrow \{1,\dots,k\}$ is injective if and only if each color class of $\phi$ is an induced star forest in $G$.

The injective chromatic index of a graph was introduced by Cardoso, Cerdeira, Cruz, and Dominic \cite{Cardoso}
in 2015 as a theoretical model for a packet radio network problem, in which the goal is to assign communication frequencies to network node pairs in a way that eliminates secondary interference.
These authors established bounds for the injective chromatic index of certain graph classes, including paths, cycles, and complete bipartite graphs. 
They also proved that computing a graph's injective chromatic index is NP-hard.
The notion of injective chromatic index was reintroduced independently in 2019 by Axenovich, D\"{o}rr, Rollin, and Ueckerdt \cite{Axenovich} under the name \emph{induced star arboricity}, and these authors proved that the injective chromatic index can be bounded in terms of a graph's treewidth or acyclic chromatic number.

Ferdjallah, Kerdjoudj, and Raspaud \cite{FKR} 
first considered the problem of bounding a graph's injective chromatic index in terms of its maximum degree. They observed that by Brooks' Theorem, a graph $G$ of maximum degree $\Delta$ satisfies $\chinj(G) \leq 2(\Delta - 1)^2$. They also observed that injective edge colorings share a close relationship with \emph{strong edge colorings}, which can be characterized as injective edge colorings in which any two incident edges receive distinct colors.
The \emph{strong chromatic index} of a graph $G$, written $\chi'_s(G)$, is the minimum number of colors required for a strong edge coloring of $G$, and hence for every graph $G$, $\chinj(G) \leq \chi_s'(G)$.
While a greedy argument shows that a graph $G$ of maximum degree $\Delta$ satisfies $\chi'_s(G) \leq 2\Delta(\Delta - 1)+1$,
Erd\H{o}s and Ne\v{s}et\v{r}il \cite{EN} conjectured that every graph $G$ of maximum degree $\Delta$ satisfies $\chi'_s(G) \leq \frac{5}{4} \Delta^2$, and this conjecture is still open. If the conjectured upper bound of Erd\H{o}s and Ne\v{s}et\v{r}il is correct, then it would be best possible, as the graph $G$ obtained from $C_5$ by replacing each vertex with an independent set of size $t$ and replacing each edge with a complete bipartite graph has maximum degree $\Delta = 2t$ and strong chromatic index exactly $\frac{5}{4} \Delta^2$. Currently, the best known upper bound for the strong chromatic index of a graph $G$ of maximum degree $\Delta$ is $\chi_s'(G) \leq 1.772 \Delta^2$, which was proven by Hurley, Kang, and de Verclos \cite{Kang} using a more general argument that applies to graphs with sparse neighborhoods.

For $d$-degenerate graphs $G$ with maximum degree $\Delta$, Miao, Song, and Yu \cite{MSY} used an edge ordering argument to show that upper bounds of the form $\chinj(G) = O(\Delta^2)$ can be greatly improved when $d$ is small:
\begin{theorem}[\cite{MSY}]
\label{thm:MSY}
    If $G$ is a $d$-degenerate graph of maximum degree $\Delta$, then
    \[\chinj(G) \leq (4d-3) \Delta-2d^2 - d + 3.\]
\end{theorem}

One frequently studied class of degenerate graphs is the class of graphs with bounded Euler genus.
The \emph{Euler genus} (or \emph{nonorientable genus}) of a graph $G$ is a measure of the complexity of the simplest surface in which one can embed $G$,
and a formal definition is given in Section~\ref{sec:prelim}.
If $G$ is a graph of Euler genus $g$,
then $G$ is $O(\sqrt{g})$-degenerate, 
and hence Theorem \ref{thm:MSY} implies that $\chinj(G) = O(\Delta \sqrt{g} )$, where $\Delta$ is the maximum degree of $G$. In fact one can obtain an upper bound for $\chinj(G)$
in terms of $g$ that is independent of $\Delta$.
Indeed, Axenovich et al.~\cite{Axenovich} show that given a graph $G$, 
\begin{eqnarray}
\label{eqn:XiXa}
\log_3(\chi_a(G)) \leq \chinj(G) \leq 3 \binom{\chi_a(G)}{2},
\end{eqnarray}
where $\chi_a(G)$ is the acyclic chromatic number of $G$.
Alon, Mohar, and Sanders \cite{AMS} proved that if $G$ is a graph of Euler genus $g$, then $\chi_a(G) = O(g^{4/7})$, implying that $\chinj(G) = O(g^{8/7})$.

Using a probabilistic approach, 
Kostochka, Raspaud, and Xu \cite{KRX} showed that
given a graph $G$ of maximum degree $\Delta$, upper bounds of the form $\chinj(G) = O(\Delta^2)$ can also be greatly improved when $G$ has small chromatic number. 
\begin{theorem}[\cite{KRX}]
\label{thm:KRX}
If $G$ is a graph of maximum degree $\Delta$ and chromatic number $\chi$, then \[\chinj (G) \leq   (\chi-1) \lceil 27\Delta \log \Delta\rceil .\]
\end{theorem}
The main idea of the proof of Theorem \ref{thm:KRX} is that given a graph $G$ and an independent set $X$, 
a certain random procedure can find an induced star forest in $G$ that contains many edges in the cut $[X,G \setminus X]$. 
By repeating this random procedure $O(\Delta \log \Delta)$ times, one can partition all edges in the cut $[X, G \setminus X]$ into $O( \Delta \log \Delta)$ induced star forests. By repeating this procedure for all but one color class in a proper coloring of $G$, one obtains an injective edge coloring of $G$.
Kostochka, Raspaud, and Xu \cite{KRX} asked whether the $\log \Delta$ factor in Theorem \ref{thm:KRX} can be removed, and this question is still open, even in the case that $G$ is bipartite.

\subsection{Background: Oriented chromatic number}
The injective chromatic index of a graph is related to the graph's \emph{oriented chromatic number}, which is defined as follows. Given an oriented graph $\vec G$, an \emph{oriented coloring} of $\vec G$ is a proper vertex coloring $\phi:V(\vec G) \rightarrow \mathbb N$ such that for each pair of arcs $uv$ and $v'u'$ in $\vec G$, $(\phi(u),\phi(v)) \neq (\phi(u'),\phi(v'))$. Then, the \emph{oriented chromatic number} of $\vec G$, written $\chi_o(\vec G)$, is the minimum number of colors required for an oriented coloring of $\vec G$. Finally, given an undirected graph $G$, the \emph{oriented chromatic number} of $G$, written $\chi_o(G)$, is the maximum value $\chi_o(\vec G)$ taken over all orientations $\vec G$ of $G$.

Raspaud and Sopena \cite{raspaud1994good} proved
that if $G$ is a graph with acyclic chromatic number $a$ and oriented chromatic number $k$, then
 $k \leq a 2^{a - 1}$,
 which shows that a graph's oriented chromatic number is bounded above by a function of its acyclic chromatic number.
 In the other direction,
 Kostochka, Sopena, and Zhu \cite{kostochka1997acyclic}
 proved that
  $a \leq k^2 + k^{3 + \log_2 \log_2 k}$, implying that a graph's acyclic chromatic number is bounded above by a function of its oriented chromatic number.
 As the inequality (\ref{eqn:XiXa}) shows that $\chinj(G)$ is bounded both above and below by unbounded increasing functions of $\chi_a(G)$,
 it follows that 
 $\chi_o(G)$ that and $\chinj(G)$ are both bounded above by functions of each other, and in particular,
\begin{eqnarray}
\label{eqn:XoUB}
\chi_o(G) \leq \chi_a(G) 2^{\chi_a(G) - 1} \leq 3^{\chinj(G)} 2 ^{3^{\chinj(G)} - 1}.
\end{eqnarray}
In fact, $\chinj(G)$ and $\chi_o(G)$ belong to a larger set of graph parameters that are bounded both above and below by unbounded increasing functions of $\chi_a(G)$, and Dvo\v{r}\'ak characterizes this set of graph parameters in \cite{Dvorak}.

Given a graph $G$ of Euler genus $g$,
the inequality (\ref{eqn:XoUB}) combined with the bound $\chinj(G) = O(g^{8/7})$ gives an upper bound for $\chi_o(G)$ in terms of $g$.
Kostochka et al.~\cite{kostochka1997acyclic} observed that
one may obtain a much smaller upper bound
by combining the inequality $\chi_o(G) \leq \chi_a(G) 2^{\chi_a(G) - 1}$
from \cite{raspaud1994good}
with a bound on the acyclic chromatic number of a graph with bounded
Euler
genus.
The current best such bound is the estimate
$\chi_a(G) = O(g^{4/7})$ of Alon, Mohar, and Sanders \cite{AMS}, which yields the upper bound $\chi_o(G) = 2^{O(g^{4/7})}$.
By generalizing the techniques of both 
\cite{AMS} and \cite{raspaud1994good},
Aravind and Subramanian \cite{aravind2009forbidden} showed that in fact 
for every constant $\epsilon >0$, $\chi_o(G) = 2^{O \left (g^{\frac{1}{2} + \epsilon}\right )}$.
Aravind and Subramanian further conjectured that the $\epsilon$ in the exponent can be removed as follows:
\begin{conjecture}[\cite{aravind2009forbidden}]
\label{conj:AS}
    If $G$ is a graph of Euler genus $g$, then $\chi_o(G) =   2^{O\left (\sqrt g \right )}$.
\end{conjecture}

We note that all previous upper bounds on the oriented chromatic number of a graph $G$ in terms of its Euler genus use the following proof strategy, originally appearing in \cite{kostochka1997acyclic}.
First, a proper $k$-coloring $\psi$ of $G$ is fixed, such that the two-colored subgraphs of $G$ under $\psi$ satisfy some specific property, such as being acyclic or having components with few edges.
Next, each vertex $v \in V(G)$ receives a new color $\psi'(v)$ which is one of $f(k) \geq 2^{k-1}$ possible \emph{shades} of the color $\psi(v)$, where the precise value of $f(k)$ depends on the properties of $\psi$.
Finally, it is argued that the new coloring $\psi'$ is an oriented coloring of $G$, which implies that $\chi_o(G) \leq kf(k)$.
We observe that if $G$ has Euler genus $g$, then a proper coloring $\psi$ of $G$ may require $k = \Omega(\sqrt g)$ colors.
As the strategy outlined above
requires at least $2^{k-1}$ shades to be allowed for each vertex,
the bound in 
Conjecture \ref{conj:AS}
is best possible using this approach.

\subsection{Our results}
In this paper we will prove upper bounds for the injective chromatic index of graphs of bounded degeneracy, and in particular, graphs of bounded Euler genus. Then, we will extend our techniques to prove a
new upper bound for the oriented chromatic number of a graph in terms of its Euler genus.

In Section \ref{sec:deg}, we prove the following result, which greatly improves Theorems \ref{thm:MSY} and \ref{thm:KRX} for graphs of small degeneracy:
\begin{theorem}
\label{thm:degen_intro}
If $G$ is a d-degenerate graph of maximum degree $\Delta$, then
\[\chinj(G) = O(d^3 \log \Delta).\]
\end{theorem}
We note that 
the factor of $\log \Delta$ in Theorem \ref{thm:degen_intro} cannot be entirely avoided. 
In particular, we will show that if $G$ is a graph, and if $G'$ is obtained from $G$ by subdividing each edge exactly once, then $\chinj(G) = \Theta(\log \chi(G))$. Thus, if $G$ has maximum degree $\Delta$ and $\chi(G) \geq \Delta^c$ for some constant $c > 0$, then $G'$ is a $2$-degenerate graph with injective chromatic index $\chinj(G') = \Theta(\log \Delta)$, which shows that the $\log \Delta$ factor in Theorem \ref{thm:degen_intro} is necessary for small values of $d$.

In Section \ref{sec:gen}, 
we consider graphs of bounded Euler genus,
and we prove the following theorem:
\begin{theorem}
\label{thm:g_intro}
If $G$ is a graph of Euler genus $g$, then 
\[\chinj(G) \leq (3+ o(1))g.\]
\end{theorem}
Theorem \ref{thm:g_intro}
significantly improves the upper bound of $\chinj(G) = O(g^{8/7})$ which follows 
from considering the acyclic chromatic number of $G$.
Furthermore,
when $n \geq 8$, the complete graph on $n$ vertices has Euler genus $\lceil \frac{(n-3)(n-4)}{6} \rceil$ 
\cite[Equation (4.19)]{RY}
and injective chromatic index $\binom{n}{2}$, which shows that Theorem \ref{thm:g_intro} is tight up to the $o(1)$ function.

In Section \ref{sec:oriented}, we 
show that the oriented chromatic number of a graph $G$ is at most exponential in its injective chromatic index,
improving the inequality (\ref{eqn:XoUB}).
We combine this fact with our methods from Section \ref{sec:gen} to prove the following theorem, 
which significantly improves the upper bound of $\chi_o(G) = 2^{O(g^{\frac{1}{2}+\epsilon})}$ given by Aravind and Subramanian \cite{aravind2009forbidden} 
for a graph of Euler genus $g$,
and we will hence show that Conjecture \ref{conj:AS} is correct.
\begin{theorem}
\label{thm:g_oriented_intro}
    If $G$ is a graph of sufficiently large Euler genus $g$, then $\chi_o(G) \leq g^{6400}$.
\end{theorem}

We also give an example of a graph $G$ of each Euler genus $g \geq 2$ satisfying
$\chi_o(G) = \Omega \left ( (g^2 / \log g)^{1/3} \right )$,
which shows that 
the $6400$ exponent in Theorem \ref{thm:g_oriented_intro} is within a factor of less than $10000$ of being correct.

Finally, in Section~\ref{sec:Improving_oriented} we provide one possible avenue for improving the bound in Theorem~\ref{thm:g_oriented_intro} involving the \emph{$2$-dipath chromatic number} of a graph, which is defined as follows.
Given an oriented graph $\vec G$,
a \emph{$2$-dipath coloring} of $\vec G$ is a coloring $\phi:V(G) \rightarrow \mathbb N$ 
such that for each pair of vertices $u,v \in V(\vec G)$ joined by a directed path of length $1$ or $2$,
$\phi(u) \neq \phi(v)$. We write $\chi_2(\vec G)$ for the minimum integer $k$ for which there exists 
a $2$-dipath coloring $\phi:V(G) \rightarrow \{1,\dots,k\}$. Then, given an undirected graph $G$, we say that the \emph{$2$-dipath chromatic number of $G$}, written $\chi_2(G)$,
is the maximum value $\chi_2(\vec G)$ taken over all orientations $\vec G$ of $G$.
Since every oriented coloring of an oriented graph is also a $2$-dipath coloring, it follows that $\chi_2(G) \leq \chi_o(G)$ for every graph $G$.

By combining ideas from Section~\ref{sec:gen} and Section~\ref{sec:oriented} with a recent technique of Clow and Stacho \cite{clow2023oriented},
we show that
it is possible to simplify 
analysis of the oriented chromatic number
of graphs of bound genus. 
In particular, we prove the following.

\begin{theorem}
\label{thm:g_improvement}
    If $G$ is a graph of Euler genus $g \geq 0$ satisfying $\chi_2(G) = k$, then 
    \[\chi_o(G) = O (k \log k + g).\]
\end{theorem}

As mentioned previously, $\chi_2(G) \leq \chi_o(G)$ for all graphs $G$. Hence, if
$k$ is the minimum value for which $\chi_2(G) = O(g^k)$ for all graphs $G$ of Euler genus $g$ and $k \geq 1$,
then the bound $\chi_o(G) = O(g^k \log g)$ also holds,
and this upper bound is tight up to the $\log g$ factor.
This fact is surprising,
given that 
bounds for $\chi_2(G)$ and $\chi_o(G)$ 
based on other parameters of $G$
are often extremely far apart.
For example, Kostochka, Sopena, and Zhu \cite{kostochka1997acyclic} showed that there exist graphs $G$ of maximum degree $\Delta$ satisfying $\chi_o(G) \geq 2^{\Delta/2}$, while a greedy argument shows that $\chi_2(G) \leq \Delta(\Delta-1) + 1$. 

Furthermore, 
this connection between $\chi_2(G)$ and $\chi_o(G)$ for graphs of Euler genus 
$g$
seems to give an encouraging path for future efforts to bound $\chi_o(G)$ in terms of $g$,
as $2$-dipath colorings are generally easier to construct than oriented colorings.
One particular reason is that 
the constraints of a $2$-dipath coloring are local,
whereas oriented colorings must satisfy global properties.
Since directly constructing oriented colorings of graphs is often difficult,
bounds on $\chi_o(G)$ are often
computed using inequalities involving other parameters of $G$, 
as in Theorem \ref{thm:g_improvement}.

\subsection{Notation and preliminaries}
\label{sec:prelim}
Let $Z$ be an event in a probability space which depends on some parameter $x$. If $\lim_{x \rightarrow \infty} \Pr(Z) = 1$, then we say that $Z$ occurs \emph{asymptotically almost surely}, or \emph{a.a.s.}~for short.

If $G$ is an oriented graph, then we write $A(G)$ for the set of arcs in $G$.
If $v \in V(G)$, then we say that $w$ is an \emph{out-neighbor} of $v$ if $vw \in A(G)$. We write $N^+(v)$ for the set of out-neighbors of $v$, and we write $A^+(v)$ for the set of arcs in $G$ outgoing from $v$.
We write $\deg^+(v) = |A^+(v)| = |N^+(v)|$, and we call this quantity the \emph{out-degree} of $v$. We write $\Delta^+(G)$ for the maximum value $\deg^+(v)$ taken over all vertices $v \in V(G)$.
Given a vertex subset $U \subseteq V(G)$, we write $A^+(U) = \bigcup_{v \in U} A^+(v)$.

A \emph{surface} is a 
connected compact Hausdorff space which is locally homeomorphic to an open disc in the plane. 
It is well known (see e.g.~\cite[Section 3]{MoharThomassen}) that for every surface $S$, there exists a graph $G$ with a \emph{$2$-cell embedding in $G$}, that is, an embedding $\pi:G \rightarrow S$ such that each connected component of $S \setminus \pi(G)$ is homeomorphic to an open disc in the plane.
If $S$ is a surface and $G$ is a graph with $n$ vertices and $e$ edges which has an $2$-cell embedding in $G$ with $f$ faces, then the \emph{Euler genus} of $S$ is the quantity $2 - n + e - f$.
The \emph{Euler genus} of a graph $G$ is the minimum value $g$ such that $G$ has an embedding on a surface of Euler genus $g$. 
In particular, a planar graph has Euler genus $g = 0$, and a graph embeddable on the projective plane has Euler genus $g \leq 1$, and a graph embeddable on the torus or Klein bottle has Euler genus $g \leq 2$.
It is well known that if $G$ has an Euler genus of $g$, then $G$ has a $2$-cell embedding on a surface of Euler genus $g$.

\section{The injective chromatic index of degenerate graphs}
\label{sec:deg}
The goal of this section is to prove Theorem \ref{thm:degen_intro} and also to show that the $\log \Delta$ factor in Theorem \ref{thm:degen_intro} cannot be removed. 
The main tool that we will use to prove Theorem \ref{thm:degen} is a random method similar to the one used by Kostochka, Raspaud, and Xu \cite{KRX} to prove Theorem \ref{thm:KRX}. The idea behind our random method is
that first, if $G$ is a $d$-degenerate graph, then we may give $E(G)$ an acyclic orientation of maximum out-degree $d$. Then, if $X$ is an independent set in $G$, a randomly chosen subset $S \subseteq V(G) \setminus X$ will have the property that many vertices $x \in X$ have a unique out-neighbor $a \in S$. This property of $S$ will allow us to find large induced star forests in $G$ and thereby partition many edges of $G$ into few induced star forests.

As the chromatic number of a $d$-degenerate graph is at most $d+1$, the following theorem implies Theorem \ref{thm:degen_intro}.

\begin{theorem}
\label{thm:degen}
If $G$ is a d-degenerate graph of maximum degree $\Delta \geq 3$ and chromatic number $\chi$, then
\[\chinj(G) \leq \lceil 4  ed \log \Delta  \rceil(2d+1) \chi.\]
\end{theorem}

One tool that we will need to prove Theorem \ref{thm:degen} is the following form of the Lov\'asz Local Lemma.
\begin{lemma}[\cite{LLL}]
\label{lem:LLL}
    Let $\mathcal B$ be a set of bad events in a probability space, and suppose that each event $B \in \mathcal B$ is mutually independent with all but fewer than $D$ other events in $\mathcal B$.
    If each event $B  \in \mathcal B$ has probability at most $p$, and if $ep D \leq 1$, then with positive probability, no event in $\mathcal B$ occurs.
\end{lemma}

With the Lov\'asz Local Lemma, we 
are ready to establish our main 
tool
for proving Theorem \ref{thm:degen}.

\begin{lemma}
\label{lem:d2logD}
Let $G$ be an oriented graph with 
maximum degree $\Delta \geq 3$ and 
maximum out-degree $d$.
If $X$ is an independent set in $G$, 
then $A^+(X)$
can be partitioned into
\[\lceil 4  ed \log \Delta  \rceil (2d+1) \]
star forests which are induced in $G$.
\end{lemma}
\begin{proof}
We iterate the following random procedure for $c = 1,2, \dots,  \lceil 4 ed  \log \Delta \rceil $:
    \begin{quote}
        Define a random subset $S_c \subseteq V(G) \setminus X$ by adding each vertex $a \in V(G) \setminus X$ to $S_c$ independently with probability $1/d$. Then, for each $x \in X$, if $N^+(x) \cap S_c$ contains exactly one vertex $a$, then color $xa$ with the color $c$.
    \end{quote}
    If our procedure instructs us to color an edge $e$ that has already been colored, then we consider the old color of $e$ to be overwritten by the new color.

    We claim that with positive probability, our procedure colors all arcs in $A^+(X)$.
    For this, we will use the Lov\'asz Local Lemma (Lemma \ref{lem:LLL}).
    For each vertex $x \in X$ and out-neighbor $a \in N^+(x)$, we define the bad event $B_{xa}$ to be the event that our procedure does not color the arc $xa$. If we can show that with positive probability no bad event occurs then this claim holds.
    To this end,
    consider a vertex $x \in X$ and an out-neighbor $a \in N^+(x)$. On each iteration of our procedure, $xa$ is colored if and only if $a \in S_c$ and no other out-neighbor $a' \in N^+(x)$ belongs to $S_c$. The probability that $a \in S_c$ is $\frac{1}{d}$, and given that $a \in S_c$, the probability that no other out-neighbor $a' \in N^+(x)$ belongs to $S_c$ is $(1-\frac{1}{d})^{\deg^+(x)-1} \geq (1 - \frac{1}{d})^{d-1} > \frac{1}{e}$. Hence, on each iteration, $xa$ is colored with probability at least $\frac{1}{ed}$, and hence the probability that $xa$ is never colored during this procedure is at most $(1 - \frac{1}{ed})^{4 ed  \log \Delta } < \Delta^{-4}$.  
    
    Next, we count the dependencies of each bad event $B_{xa}$.
    We observe that the event $B_{xa}$ is determined entirely by independent random events at $N^+(x)$. 
    Therefore, $B_{xa}$ is mutually independent with all sets of bad events $B_{yb}$ (for $y \in X$ and $b \in V(G) \setminus X$) where $\dist(x,y) > 3$. The number of vertices $y \in X$ within distance $2$ of $x$ is less than $\Delta^2$,
    and each vertex $y$ has at most $\Delta$ out-neighbors $b \in N^+(y)$. Therefore, $B_{xa}$ is mutually independent with all but fewer than $\Delta^3$ other bad events $B_{yb}$. Since $ \Delta^{-4} \cdot \Delta^3 \cdot e < 1$ for $\Delta \geq 3$, it follows from the Lov\'asz Local Lemma (Lemma \ref{lem:LLL}) that
    with positive probability, no bad event $B_{xa}$ occurs. Thus, we will assume that our coloring procedure assigns a color to each arc of $A^+(x)$.

    Now, for each color $c$ used during our procedure,
    we recolor each arc of color $c$ with an ordered pair of the form $(c,i)$, as follows.
    For each color $c$, 
    we define a directed graph $H_c$ on the vertex set $S_c$. For two vertices $a,a' \in S_c$, we say that $aa' \in A(H_c)$ if and only if  either $aa' \in A(G)$ or     $G$ contains a (not necessarily colored) arc $ax$ and an arc $x a'$ of color $c$. Since each vertex $a \in S_c$ has out-degree at most $d$, and since each vertex $x \in X$ has at most one outgoing edge of color $c$, it follows that $H_c$ has maximum out-degree at most $d$. Hence, the maximum average degree of $H_c$ is at most $2d$, so $H_c$ can be properly colored with $2d+1$ colors. Let $\phi_c$ be a proper coloring of $H_c$ with $2d+1$ colors. For each arc $xa$ of color $c$ in $G$, we give $xa$ the new color $(c,\phi_c(a))$.

    Next, consider two arcs $x_1a_1, x_2a_2 \in A(G)$ at distance $1$ which originally received the color $c$. Since each vertex $x \in X$ contained in an arc of color $c$ has exactly one out-neighbor in $S_c$, it follows that $x_1 a_1$ and $x_2 a_2$ must be joined without loss of generality either by the arc $a_1 x_2$ or the arc $a_1 a_2$. Then, $H_c$ must contain the arc $a_1 a_2$, and hence $\phi_c(a_1) \neq \phi_c(a_2)$. 
    Therefore, $x_1 a_1$ and $x_2 a_2$ receive different colors after recoloring, and thus each color class created by our procedure is an induced star forest in $G$.
    As our coloring procedure uses at most $\lceil 4  ed \log \Delta  \rceil (2d+1)$ colors, the proof is complete.
    \end{proof}

Now, the proof of Theorem \ref{thm:degen} follows easily.
\begin{proof}[Proof of Theorem \ref{thm:degen}]
Let $G$ be a $d$-degenerate graph. Give $G$ a proper $\chi$-coloring $\psi$, and apply Lemma \ref{lem:d2logD} to each color class $X$ of $\psi$ to produce an edge coloring $\phi$ of $G$.
Since each arc $e \in A(G)$ belongs to $A^+(X)$ for some color class $X$ of $G$,
$\phi(e)$ is defined for every edge $e \in E(G)$.
As each color class of $\phi$
is an induced star forest in $G$,
$\phi$ is
an injective edge coloring of $G$ with at most $\lceil 4  ed \log \Delta  \rceil (2d+1) \chi$ colors.
\end{proof}

To prove that the $\log \Delta$ factor in Theorem \ref{thm:degen} cannot be entirely avoided, we use the following lemma and proposition.

\begin{lemma}\cite[Exercise 5.1.32]{West}
    \label{lem:west}
    A graph $G$ is $2^k$-colorable if and only if $G$ is the union of $k$ bipartite graphs.
\end{lemma}

\begin{proposition}
\label{prop:log}
Let $G$ be a graph, and let $G'$ be obtained from $G$ by subdividing each edge exactly once. Then $\chinj(G') \geq \log_2 \chi(G)  $.
\end{proposition}
\begin{proof}
    We consider $G$ to be a directed graph in which each edge is bidirectional.
    Also, for ease of presentation, we consider $V(G)$ to be a subset of $V(G')$ in the natural way.
    Now, consider an injective edge-coloring $\phi$ of $G'$, and let $c$ be one of the colors used by $\phi$. 
    We represent the edges in $G'$ colored by $c$ with a directed subgraph $H_c$ of $G$ as follows. 
    If an edge $e \in E(G')$ of color $c$ is incident to a vertex $u \in V(G)$ and is at distance $1$ from a vertex $v \in V(G)$, we add an arc in $H_c$ from $u$ to $v$. Since $\phi$ is an injective edge-coloring, $H_c$ has no directed path of length $2$
    and hence is bipartite. Therefore, the color classes of $\phi$ correspond to directed bipartite subgraphs of $G$ whose union gives the entire arc set of $G$. By Lemma \ref{lem:west}, 
    $\phi$ uses at least $\log_2 \chi(G)$ colors, completing the proof.
\end{proof}
We note that a similar argument shows that $\chinj(G') \leq 2 \lceil \log_2 \chi(G) \rceil$.

By Proposition \ref{prop:log}, if $G$ is a graph with maximum degree $\Delta$ and chromatic number $\Delta^c$ for some constant $c > 0$, then $G'$ is a $2$-degenerate graph satisfying $\chinj(G') \geq c \log \Delta$.
Therefore, the $\log \Delta$ factor in Theorem \ref{thm:degen} cannot be removed when $d$ is small.

\section{The injective chromatic index of graphs of bounded genus}
\label{sec:gen}
In this section, we aim to 
prove Theorem \ref{thm:g_intro},  
which states that if $G$ is a graph of Euler genus $g$, then $\chinj(G) \leq (3 + o(1))g$.
If $G$ has Euler genus at most $1$, then $G$ has acyclic chromatic number at most $7$ \cite{AMS},
and hence
the inequality (\ref{eqn:XiXa})
tells us that $\chinj(G) \leq 63$.
As our goal is to obtain an asymptotic bound on the injective chromatic index of $G$ in terms of  $g$, we will assume in this section that $g \geq 2$, so that $g$ and $\log g$ are both positive.

To prove an upper bound on the injective chromatic index of a graph in terms of its Euler genus, we use an approach similar to the one in the proof
of Lemma \ref{lem:d2logD}. 
Recall that in Lemma \ref{lem:d2logD},
we consider a graph $G$ and an independent set $X \subseteq V(G)$, and we have a random procedure that aims to choose sets $S_c \subseteq V(G) \setminus X$ with the property that for many vertices $x \in X$, $|N^+(x) \cap S_c| = 1$.
One disadvantage of this method is that since the random approach relies on the Lov\'asz Local Lemma, the upper bound that we obtain from this method depends on the maximum degree of the graph.

In this section, we will develop a deterministic version of this procedure that often gives us bounds that do not depend on the maximum degree of a graph, as follows.
Suppose we have a graph $G$ and an independent set $X \subseteq V(G)$. Rather than adding vertices of $V(G) \setminus X$ to a set $S_c$ independently at random, 
we first give $V(G) \setminus X$ a coloring $\phi$ so that for each $x \in X$, the vertices in $N^+(x)$ all receive distinct colors. Then, 
we construct
sets $S_c \subseteq V(G) \setminus X$ with the property that $|N^+(x) \cap S_c| = 1$ for many vertices $x \in X$ by considering the colors of the vertices in $V(G) \setminus X$. 
By using this approach, we
remove dependence on maximum degree from the upper bounds that we obtain.
We note that Dvo\v{r}\'ak \cite{Dvorak}
uses a similar technique 
to characterize graph classes with bounded acyclic chromatic number.

First, we establish a lemma which will help us find colorings $\phi$ satisfying the properties described above.

\begin{lemma}
    \label{lem:partitions}
    Let $k \geq r \geq 2$ be positive integers. There exists a family of $\lceil er^2  \log k   \rceil $ subsets $P_i \subseteq \{1, \dots, k\}$ 
    such that for each sequence $(a_1, \dots, a_{\ell})$ of $\ell \leq r$ distinct elements from $\{1, \dots, k\}$, there exists a subset $P_i$ such that $a_1 \in P_i$ and $a_2, \dots, a_{\ell} \not \in P_i$.
\end{lemma}
\begin{proof}
    It is enough to prove the lemma under the assumption that $\ell = r$.    
    For $1 \leq i \leq \lceil er^2 \log k  \rceil$,
    we construct a subset $P_i$
    by adding each $j \in \{1,\dots,k\}$ to $P_i$ independently with probability $\frac{1}{r}$.
    We note that the subsets $P_i$ in our family may not all be distinct, but this is not a problem.
    For a given sequence $(a_1, \dots, a_{r})$, the probability that $a_1 \in P_i$ and $a_2, \dots, a_{r}  \not \in P_i$ is equal to $\frac{1}{r}(1-\frac{1}{r})^{r-1} > \frac{1}{er}$. 
    Therefore, the probability that these conditions do not hold for any subset $P_i$
    is less than 
    \[\left (1 - \frac{1}{er} \right )^{\lceil er^2 \log k  \rceil } < \frac{1}{k^r}.\]  
    As the number of sequences $(a_1, \dots, a_r)$ with $ r$ distinct elements from $\{1,\dots,k\}$ is less than $k^r $,
    the expected number of such sequences 
    that do not satisfy our property for some $P_i$
    is less than $1$. Thus, with positive probability, our subsets $P_i$ satisfy the lemma.
\end{proof}

Now, using Lemma \ref{lem:partitions}, we can
carry out a deterministic analogue of the random procedure from Lemma \ref{lem:d2logD}.
Using our new deterministic procedure, 
we establish the following lemma, which will
help us find bounds for the injective chromatic indices of graphs which do not depend on maximum degree.
\begin{lemma}
\label{lem:colorclass}
    Let $G$ be an oriented graph, and let $X \subseteq V(G)$ be an independent set in $G$ with maximum out-degree $d$. Let $H$ be a graph on $V(G) \setminus X$ defined so that two vertices $u,v \in V(G) \setminus X$ are adjacent if and only if there exists a vertex $x \in X$ such that $u,v \in N^+(x)$. Then, the set of arcs in $G$ outgoing from $X$ can be partitioned into 
    \[ \lceil ed^2 \log \chi(H)  \rceil( 2\Delta^+(G) + 1 ) = O(d^2 \Delta^+(G) \log \chi(H))\]
    star forests which are induced in $G$.
\end{lemma}
\begin{proof}
    Let $\phi$ be a proper coloring of $H$ with $k = \chi(H)$ colors. %
    Let $\mathcal P$ be a set of $\lceil ed^2 \log k  \rceil$ subsets $P_i \subseteq \{1,\dots,k\}$ such that for each sequence $(a_1,\dots,a_{\ell})$ of $\ell \leq r$ distinct elements from $\{1,\dots,k\}$, there exists a subset $P_i$ such that $a_1 \in P_i$ and $a_2,\dots,a_{\ell} \not \in P_i$.
    The set $\mathcal P$ exists by Lemma \ref{lem:partitions}.

    Now, we color $A^+(X)$ as follows. For each subset $P_i \in \mathcal P$,
    we execute the following steps. 
    \begin{enumerate}
        \item Initialize sets $V_i = \emptyset$, $E_i = \emptyset$.
        \item Define $X_i \subseteq X$ as the set of vertices $x \in X$ such that exactly one vertex $v \in N^+(x)$ satisfies $\phi(v) \in P_i$.
        \item 
        For each $x \in X_i$, let $z_i(x)$ be the unique out-neighbor of $x$ for which $\phi(z_i(x)) \in P_i$.
        Update $V_i \leftarrow V_i \cup \{z_i(x)\}$ and  $E_i \leftarrow E_i \cup \{x z_i(x)\}$.
        \item Define an oriented graph $D_i$ with vertex set $V_i$ and with arcs defined as follows. 
    For any two vertices $u,v \in V_i$, add the arc $uv$ to $D_i$ if and only if either $uv \in A(G)$ or there exists a vertex $x \in X_i$ 
    such that $ux \in A(G)$ and $v = z_i(x)$. 
    \item \label{step:color} Give $D_i$ a proper coloring $\psi_i$ using the set $\{1, \dots, 2\Delta^+(G) + 1\}$. 
    \item For each arc $xv \in E_i$, color $xv$ with the color $(i,\psi_i(v))$.
    \end{enumerate}
    Again, if our procedure asks us to color some arc $e \in A(G)$ that has already been colored, then we let the new color of $e$ replace the old color.
    Note that since each vertex $x \in X_i$ has a unique out-neighbor $z_i(x)$, 
    it follows that for each $v \in V_i$, $\deg^+_{D_i}(v) \leq \deg^+_G(v)$, so $D_i$ has maximum out-degree at most $\Delta^+(G)$. Hence, $D_i$ is a $2 \Delta^+(G)$-degenerate graph, and thus Step (\ref{step:color}) is possible.

    First, we claim that the procedure above colors each arc of $A^+(X)$. Indeed, consider a vertex $x \in X$ and an out-neighbor 
    $u \in N^+(x)$. 
    Write $N^+(x) = \{u, w_1, \dots, w_t\}$, and consider the sequence $S = (\phi(u), \phi(w_1),\dots,\phi(w_t))$. Since $N^+(x)$ induces a clique in $H$, all elements in $S$ are distinct. 
    Hence, there exists a subset $P_i \in \mathcal P$
    such that $\phi(u) \in P_i$ and $\phi(w_1),\dots,\phi(w_t) \not \in P_i$. Then, by construction, $xu$ is colored with some color $(i,\psi_i(u))$.

    Next, we claim that in our 
    coloring of $A^+(X)$, each color class is a star forest which is induced in $G$. Indeed, consider two arcs $xu$ and $yv$ in $A(G)$ at distance $1$ in $G$, where $x,y \in X$ and $u,v \in V(G) \setminus X$.
    Suppose that both $xu$ and $yv$ are colored with the color $(i,j)$.
    If this occurs, then it must hold that $u = z_i(x)$, $v = z_i(y)$, and $u, v \in V_i$.
    Since $xu$ and $yv$ are at distance $1$ in $G$, and since $X$ is an independent set, one of the following cases must hold without loss of generality.
    \begin{enumerate}
        \item $yu \in A(G)$.
        Since
        $v = z_i(y)$,
        $\phi(v) \in P_i$ and $\phi(u) \not \in P_i$. Then $u \not \in V_i$, a contradiction.
        \item $uy \in A(G)$.
        As $v = z_i(y)$, 
        $uv$ is an arc of $D_i$. Hence, $\psi_i(u) \neq \psi_i(v)$, and $xu$ and $yv$ cannot both be colored with $(i,j)$, a contradiction.
        \item $uv \in A(G)$.
        Then, as before, 
        $uv$ is an arc of $D_i$, and $\psi_i(u) \neq \psi_i(v)$, which again gives a contradiction.
    \end{enumerate}
    Therefore, each color class in $A^+(X)$ 
    produced by our procedure is an induced star forest in $G$. Since each color class is of the form $(i,j)$, where $i \in \{1,\dots,\lceil ed ^2 \log k  \rceil\}$ and $j \in \{1,\dots, 2\Delta^+(G) + 1\}$, we use 
    at most $ \lceil ed^2 \log k  \rceil( 2\Delta^+(G) + 1 )$
    colors in our coloring. This completes the proof.
\end{proof}

While Lemma \ref{lem:colorclass}
does give us a tool for partitioning the edges of an oriented graph $G$ into induced star forests, it is unclear how to estimate the chromatic number of the graph $H$ defined in the lemma statement
in general.
In the following lemmas, we will
show that if $G$ is an oriented graph of Euler genus $g$, then 
whenever we apply Lemma \ref{lem:colorclass} to an independent set $X \subseteq V(G)$,
we can bound the value $\chi(H)$ by a function of $g$ and $\Delta^+(G)$.

Given a hypergraph $\mathcal H$, the \emph{Levi graph} of $\mathcal H$ is the bipartite graph $L$ with vertex set $V(\mathcal H) \cup E(\mathcal H)$
such that for each $v \in V(\mathcal H)$ and $e \in E(\mathcal H)$, $ve \in E(L)$ if and only if $v \in e$.
Adopting a standard convention (see e.g.~\cite{JM}),
we say that $\mathcal H$ has Euler genus $g$ if and only if the Levi graph of $\mathcal H$ has Euler genus $g$.
Additionally,
we define
the \emph{clique graph} $K(\mathcal H)$ of $\mathcal H$ 
as the graph on $V(\mathcal H)$ such that two vertices $u,v \in V(\mathcal H)$ are adjacent in $K(\mathcal H)$ if and only if $u$ and $v$ belong to a common edge of $\mathcal H$. Observe that $K(\mathcal H)$ is formed by replacing each edge of $\mathcal H$ with a clique.
The notion of a clique graph will be useful to us, as the graph $H$ in the statement of Lemma \ref{lem:colorclass}
can be defined as the clique graph of a certain hypergraph.

In the following lemmas, we establish an upper bound for the chromatic number of a clique graph obtained from a hypergraph of Euler genus $g$.

\begin{lemma}
\label{lem:mindeg}
    If $\mathcal H$ is a hypergraph with edges of size at most $r$ and Euler genus at most $g \geq 2$, then $K(\mathcal H)$ has a vertex of degree at most $20 r^2 \sqrt g - 1$.
\end{lemma}
\begin{proof}
For our proof, we assume that each edge in $\mathcal H$ contains at least two vertices, as edges of size one have no influence on $K(\mathcal H)$.
We write $L$ for the Levi graph of $\mathcal H$.
We partition $E(\mathcal H)$ into parts $E_2$ and $E_{\geq 3}$, where $E_2$ consists of all edges in $E(\mathcal H)$ of size $2$, and $E_{\geq 3}$ contains all other edges of $\mathcal H$. We aim to find upper bounds for $E_2$ and $E_{\geq 3}$.

First, we consider the edge set $E_2$. The graph $H_2 = (V(\mathcal H), E_2)$ is a topological minor of $L$, so the Euler genus of $H_2$ is at most $g$.
Therefore, by Euler's formula, $|E_2| \leq 3 |V(\mathcal H)| + 3g - 6$.

Next, we consider the edge set $E_{\geq 3}$. Let $L_{\geq 3}$ be the Levi graph of the hypergraph $(V(\mathcal H), E_{\geq 3})$, and consider an embedding of $L_{\geq 3}$ on a surface of minimum Euler genus.
Since $L_{\geq 3}$ is a subgraph of $L$, the Euler genus of $L_{\geq 3}$ is at most $g$. Thus, by Euler's formula,
    \begin{eqnarray*}
    |V(L_{\geq 3})| - |E(L_{\geq 3})| + |F(L_{\geq 3})| &\geq& 2 - g \\
    |V(L_{\geq 3})| - \frac{1}{2} |E(L_{\geq 3})| &\geq& 2 - g \\
    |V(\mathcal H)| + |E_{\geq 3}| - \frac{3}{2} |E_{\geq 3}| &\geq& 2-g \\
    |E_{\geq 3}| &<&  2|V(\mathcal H)|  + 2g
\end{eqnarray*}

Now, let $u_1, \dots, u_{|E(\mathcal H)|}$ be the vertices in $L$ corresponding to the edges of $\mathcal H$. We observe that 
\begin{eqnarray*}
|E(K(\mathcal H))| &\leq& \binom{\deg u_1}{2} + \binom{ \deg u_2}{2} + \dots + \binom{\deg u_{|E(L)|}}{2} \\
 &<& (\deg u_1)^2 + (\deg u_2)^2 + \dots + (\deg u_{|E(\mathcal H)|})^2 \\
 &\leq &  r^2 |E(\mathcal H)| = r^2 (|E_2| + |E_{\geq 3}|) \\
 & < & r^2 (5|V(\mathcal H)| + 5g ) .
\end{eqnarray*}
Now, if $|V(\mathcal H)| < \sqrt g$, then  $K(\mathcal H)$ clearly has a vertex of degree at most $20r^2 \sqrt g - 1$.
Otherwise, $|V(\mathcal H) | \geq \sqrt g$, and $K(\mathcal H)$ has a vertex of degree at most 
\[2|E(K(\mathcal H)) | / |V(\mathcal H)| < 2 r^2(5 + 5 \sqrt g) < 20 r^2 \sqrt g - 1.\]
\end{proof}
    Lemma \ref{lem:mindeg} gives us the following corollary.
\begin{lemma}
\label{lem:cliquegraph}
    If $\mathcal H$ is a hypergraph with edges of size at most $r$ and 
    Euler genus at most $g \geq 2$, then $\chi(K(\mathcal H)) \leq 20 r^2 \sqrt g $.
\end{lemma}
\begin{proof}
    Suppose the lemma is false, and let $\mathcal H$ be the hypergraph on the fewest number of vertices for which the lemma does not hold. By Lemma \ref{lem:mindeg}, $K(\mathcal H)$ has a vertex $u$ of degree at most $20 r^2 \sqrt g - 1$.
    Consider the hypergraph $\mathcal H'$ on $V(\mathcal H) \setminus \{u\}$ with edge set 
    $\{e \setminus \{u\}: e \in E(\mathcal H)\}$. If we write $L$ for the Levi graph of $\mathcal H$ and $L'$ for the Levi graph of $\mathcal H'$, clearly $L'$ is a subgraph of $L$, so $L'$ has genus at most $g$. Hence, as $\mathcal H$ is a minimum counterexample, $K(\mathcal H')$ has a proper coloring with $20 r^2 \sqrt g  $ colors. 
    Furthermore, it is easy to check that $K(\mathcal H)  \setminus \{u\} = K(\mathcal H')$. Therefore, we may properly color $K(\mathcal H)  \setminus \{u\}$ with $20 r^2 \sqrt g $ colors, and since $u$ has at most $ 20 r^2 \sqrt g - 1$ neighbors, we may extend this coloring to $u$. Hence $\mathcal H$ is in fact not a counterexample, giving a contradiction and completing the proof.
\end{proof}

Now that we have an upper bound for the chromatic number of the clique graph of a hypergraph with bounded Euler genus, we are almost ready to prove Theorem \ref{thm:g_intro}. We need one final lemma.
\begin{lemma}
\label{lem:bdd_n}
    Let $k \geq 1$ be an integer.
    If $G$ is a graph of Euler genus at most $g \geq 2$ and minimum degree at least $k + 6$, then $G$ has fewer than $\frac{6g}{k}$ vertices.
\end{lemma}
\begin{proof}
    By Euler's formula, $|V(G)| - \frac{1}{3}|E(G)| > -g$. Rearranging this,
    \[\sum_{v \in V(G)} (\deg(v) - 6) < 6g.\]
    If each vertex has degree at least $k + 6$, then the number of terms in this sum is less than $\frac{6g}{k}$, completing the proof.
\end{proof}

Now, we are ready to prove Theorem \ref{thm:g_intro}, which states
that if $G$ is a graph of Euler genus $g$, then $\chinj(G) \leq (3 + o(1))g$.

\begin{proof}[Proof of Theorem \ref{thm:g_intro}:]
As the theorem statement is asymptotic, we assume that $g$ is sufficiently large.
We write $n = |V(G)|$.
We order the vertices of $G$ as follows. We iterate through $i = n, n - 1, \dots, 3,2,1$, and on each iteration we let $v_i$ be the vertex of minimum degree in $G \setminus \{v_{i+1}, \dots, v_{n}\}$. 
A classical result of Heawood \cite[Theorem 8.3.1]{MoharThomassen}
states that $G$ has degeneracy $\left \lfloor \frac{5 + \sqrt{1 + 24 g }}{2} \right \rfloor < 3 \sqrt g - 1$,
so each vertex $v_i$ has at most $3 \sqrt{g} - 1$ neighbors appearing before $v_i$ in the ordering. We orient $E(G)$ so that $v_i v_j$ is oriented from $v_i$ to $v_j$ if and only if $i > j$. 
We also give $G$ a proper coloring $\phi$ with $3 \sqrt{g} $ colors by iterating through $i=1,\dots,n$ and coloring $v_i$ with the least available positive integer (which has not already been used at a neighbor).

Next, we partition $V(G)$ into parts $V_1 = \{v_1, \dots, v_{\lceil \frac{6g}{\log g} \rceil - 1}\}$ and $V_2 = \{ v_{\lceil \frac{6g}{\log g} \rceil }, \dots, v_{n}\}$. For each vertex $v_i \in V_2$, $v_i$ is the minimum-degree vertex in $G[v_1,\dots,v_i]$; hence, by Lemma \ref{lem:bdd_n}, for each value $i \geq \lceil \frac{6g}{\log g} \rceil$, the vertex $v_i$ has at most $\log g + 5$ neighbors $v_j$ for which $j < i$. Therefore, for each vertex $v_i \in V_2$, $\phi(v_i) \leq \log g + 6$, and $v_i$ has out-degree at most $\log g + 5$.

Now, for each value $c \in \{1,\dots,\lfloor \log g \rfloor + 6\}$, consider the independent set $X_c = \phi^{-1}(c) \cap V_2$. By the previous paragraph, each vertex in $X_c$ has out-degree at most $d = \log g + 5$. We define a bipartite graph $L_c$ on $V(G)$ consisting exactly of the edges outgoing from vertices in $X_c$. We observe that $L_c$ has genus at most $g$, and $L_c$ is the Levi graph of a hypergraph $\mathcal H_c$.
We observe further that the clique graph $K(\mathcal H_c)$ is a graph on $V(G) \setminus X_c$ defined such that two vertices $u,v \in V(G) \setminus X_c$ are adjacent in $K(\mathcal H_c)$ 
if and only if there exists $x \in X_c$ such that $u,v \in N^+(x)$. Lemma \ref{lem:cliquegraph}, $\chi(K(\mathcal H_c)) \leq 20 (\log g + 5)^2 \sqrt g$. Notice that $K(\mathcal{H}_c)$ is exactly the graph $H$ for $X_c=X$ defined in Lemma~\ref{lem:colorclass}. Let $H_c = K(\mathcal{H}_c)$. Hence, as $g$ is sufficiently large,  $\log \chi(H_c) <  \log g$.

Thus, we may apply Lemma \ref{lem:colorclass} to color all arcs outgoing from $X_c$ with $O(\log^3g \sqrt {g})$ colors, so that each color class induces a star forest in $G$. If we repeat this process for all color classes $c \in \{1,\dots,\lfloor \log g \rfloor + 6\}$, we use a total of $O(\log^4g \sqrt {g})$ colors
to partition all arcs outgoing from $V_2$ into star forests which are induced in $G$.

After coloring all arcs outgoing from $V_2$, the only uncolored edges in $G$ are those in $G[V_1]$. Hence, to finish our injective edge coloring of $G$, it suffices to use a new color for each edge of $G[V_1]$. To this end, we count the edges in $G[V_1]$.

By Euler's formula, 
    \[\sum_{v \in V_1} (\deg_{G[V_1]}(v) - 6) < 6g.\]
    Since $|V_1| < \frac{6g}{\log g}$, this implies that 
    \[2 |E(G[V_1])| = \sum_{v \in V_1} \deg_{G[V_1]}(v)  < 6g + \frac{36g}{\log g} .\]
    Hence, $G[V_1]$ has at most $(3+o(1)) g$ edges, and hence we can finish our injective edge-coloring of $G$ with at most $(3+o(1))g$ additional colors. 
    Since we used $o(g)$ colors to color the arcs of $G$ outgoing from $V_2$, we use in total $(3+o(1))g$ colors to complete an injective edge coloring of $G$.
\end{proof}

\section{The oriented chromatic number of graphs with bounded genus}
\label{sec:oriented}

In this section, we aim to prove Theorem \ref{thm:g_oriented_intro}, 
which states that the oriented chromatic number of a graph $G$ is bounded above by a polynomial function of the Euler genus of $G$. Our general strategy will be 
to show that $G$ has a certain spanning subgraph with small injective chromatic index, and then to argue that this implies an upper bound on the oriented chromatic number of $G$.

First, we establish a relationship between the injective chromatic index and the oriented chromatic number of a graph.
We observed in the inequality (\ref{eqn:XoUB}) that the oriented chromatic number of a graph is bounded above by a double exponential function of the graph's injective chromatic index. 
The following lemma shows that this upper bound can in fact be improved to an exponential function.

\begin{lemma}
\label{lem:exp}
    For every graph $G$, 
    if $\chinj(G) = k$, then $\chi_o(G) \leq 4^k$.
\end{lemma}
\begin{proof}
    Let $G$ be a graph with injective chromatic index $k$, and let $\psi:E(G) \rightarrow \{1,\dots,k\}$ be an injective edge coloring of $G$. 
    Suppose that $E(G)$ has some orientation.
    We give $V(G)$ an oriented coloring $\phi:V(G) \rightarrow 2^{\{1,\dots,k\}} \times  2^{\{1,\dots,k\}}$ by assigning each vertex $v \in V(G)$ the color $\phi(v) = (S_v^+,S_v^-)$, where $S_v^+$ is the set of colors appearing at arcs outgoing from $v$, and $S_v^-$ is the set of colors appearing at arcs going into $v$. We argue that $\phi$ is an oriented coloring.

    First, we argue that $\phi$ is proper. Indeed, suppose that there exists an arc $uv$ in $G$ such that $\phi(u) = \phi(v)$. Since $\psi(uv) \in S_v^- \cap S_u^+$, the equality $(S_u^+,S_u^-) = (S_v^+,S_v^-)$ implies that $\psi(uv) \in S_v^+ \cap S_u^-$. Then there must exist an arc $a_1$ going into $u$ of color $\psi(uv)$ as well as an arc $a_2$ outgoing from $v$ of color $\psi(uv)$, which is a contradiction, as $a_1$ and $a_2$ are either at distance $1$ or part of a common triangle. Hence, $\phi$ is proper.

    Next, suppose that there exist two arcs $uv$ and $v'u'$ in $G$ so that $\phi(u) = \phi(u')$ and $\phi(v) = \phi(v')$. Since $\psi$ is an injective coloring, it must hold that $S_u^- \cap S_v^+ = \emptyset$. Then, since $(\phi(u),\phi(v)) = (\phi(u'),\phi(v'))$, this implies that $S_{u'}^- \cap S_{v'}^+ = \emptyset$. However, this is a contradiction, since $\psi(v'u') \in S_{u'}^- \cap S_{v'}^+ $. Hence, $\phi$ is an oriented coloring.
\end{proof}

We need one more lemma about the oriented chromatic number before proving Theorem \ref{thm:g_oriented_intro}.

\begin{lemma}
    \label{lem:minusE}
Let $G$ be a graph, and let $U \subseteq V(G)$. Then $\chi_o(G) \leq |U| + \chi_o(G \setminus E(G[U]))$.
\end{lemma}
\begin{proof}
    Consider a fixed orientation of $E(G)$.
    We give $G$ a proper oriented coloring as follows. First, we define an oriented coloring $\phi$ of $G \setminus E(G[U])$ that uses $\chi_o(G \setminus E(G[U]))$ colors. Then, we define a new coloring $\psi$ by recoloring each vertex of $U$ with a new unique color. We show that $\psi$ is an oriented coloring of $G$.

    We first claim that $\psi$ is a proper coloring. Indeed, suppose that there exist two adjacent vertices $u,v \in V(G)$ so that $\psi(u) = \psi(v)$. Since $\phi$ is a proper coloring, it must follow without loss of generality that $u \in U$. However, then $u$ is the only vertex with the color $\psi(u)$, so $\psi(u) \neq \psi(v)$, a contradiction.

    Next, suppose that there exist arcs $uv$ and $v'u'$ in $G$ so that $(\psi(u), \psi(v)) =(  \psi(u'), \psi(v'))$. Since $\phi$ is an oriented coloring, it follows that one of $u,v,u',v'$ belongs to $U$. If $u \in U$, then since $u$ is the only vertex with color $\psi(u)$, it follows that $u = u'$. If $v = v'$, then $G$ contains a digon, a contradiction. Therefore, since $\psi(v) = \psi(v')$, it follows that $v,v' \not \in U$. Hence, $(\psi(u), \psi(v)) =(\psi(u'), \psi(v'))$, contradicts either assumption that $\phi$ is an oriented coloring or that $G$ contains no digon. Therefore, $\psi$ is an oriented coloring.
\end{proof}

Now, we are ready to prove Theorem \ref{thm:g_oriented_intro}, which states that 
if $G$ is a graph with sufficiently large Euler genus $g$,
then $\chi_o(G) \leq g^{6400}$.
This polynomial upper bound in $g$ gives an affirmative answer to Conjecture \ref{conj:AS} and in fact greatly improves the bound stated in this conjecture.

\begin{proof}[Proof of Theorem \ref{thm:g_oriented_intro}:]
    We let $G$ be a graph of Euler genus $g$, and we assume that $g$ is sufficiently large.
    Rather than 
    bounding $\chi_o(G)$ by considering an explicit orientation of $E(G)$,
    we will bound $\chi_o(G)$ 
    by estimating the injective chromatic index of a certain subgraph of $G$ and then
    using Lemmas \ref{lem:exp} and \ref{lem:minusE}.

We write $n = |V(G)|$, and we order $V(G)$ as follows. We iterate through $i = n, n - 1, \dots, 3,2,1$, and on each iteration we let $v_i$ be the vertex of minimum degree in $G \setminus \{v_{i+1}, \dots, v_{n}\}$. 
We then give $G$ a proper coloring $\phi$ by iterating through $i=1,\dots,n$ and coloring $v_i$ with the least available positive integer which has not already been used at a neighbor.
Next, we partition $V(G)$ into parts $V_1 = \{v_1, \dots, v_{6g}\}$ and $V_2 = \{v_{6g+1}, \dots, v_{n}\}$. For each vertex $v_i \in V_2$, $v_i$ is the minimum-degree vertex in $G[v_1,\dots,v_i]$; hence, by Lemma
\ref{lem:bdd_n}, for each value $i > 6g$, the vertex $v_i$ has at most $6$ neighbors $v_j$ for which $j < i$. Therefore, for each vertex $v_i \in V_2$, $\phi(v_i) \leq 7$. We also orient $E(G)$ so that each edge $v_i v_j$ is oriented from $v_i$ to $v_j$ if and only if $i > j$. Note that under this orientation, each vertex $v_i \in V_2$ has out-degree at most $6$. 

Now, we define $G' = G \setminus E(G[V_1])$, and we aim to bound $\chinj(G')$. 
For each color $c \in \{1,\dots,7\}$, let $X_c \subseteq V_2$ be the independent set consisting of those vertices in $V_2$ of color $c$. 
We will apply Lemma \ref{lem:colorclass} 
to partition $A^+(X_c)$ into induced star forests.
We write $\mathcal H_c$ for the hypergraph on $V(G') \setminus X_c$
with the edge set $\{N^+(x): x \in X_c\}$, and we write $k = \chi(K(\mathcal H_c))$.
By Lemma \ref{lem:cliquegraph},
$k \leq 20 \cdot 6^2 \sqrt g $.
Since $G'$ is $6$-degenerate and has maximum out-degree $6$,
Lemma \ref{lem:colorclass} tells us that $A^+(X_c)$ can be partitioned into
$13 \lceil 36e \log k \rceil \leq (234e + o(1)) \log g$ star forests which are induced in $G'$.
 By repeating this process for all $7$ color classes of $G'$, we find an injective edge-coloring of $G'$ using at most $(1638e + o(1)) \log g$ colors. 

Finally, by Lemma \ref{lem:exp}, $\chi_o(G') \leq 4^{(1638e + o(1)) \log g} < g^{6400} - 6g$ for large $g$. Since $G' = G \setminus E(G[V_1])$, it then follows from Lemma \ref{lem:minusE}
that $\chi_o(G) \leq \chi_o(G') + |V_1| \leq g^{6400}$, completing the proof.
\end{proof}

We conclude this section 
with a random construction which shows the existence of oriented graphs with large Euler genus $g$ and oriented chromatic number at least $g^{\frac{2}{3} - o(1)}$.
Rather than directly estimating the oriented chromatic number of a random construction $G$, we instead 
consider its $2$-dipath chromatic number $\chi_2(G)$.
The random construction that we use 
is a standard method 
for constructing graphs for which various coloring parameters are large,
such as acyclic chromatic number \cite{AMR,AMS}, star chromatic number \cite{FRR}, and 
nonrepetitive chromatic number \cite{AlonNR}.
This construction shows us that the exponent $6400$ in Theorem \ref{thm:g_oriented_intro} is correct within a factor of less than $10000$.

\begin{proposition}\label{prop: oriented_lower}
    There exists a constant $c > 0$ such that for each value $g \geq 2$, there exists an oriented graph $G$ of Euler genus $g$ for which $\chi_o(G) \geq \chi_2(G)\geq   c\left ( \frac{g^2}{\log g} \right )^{1/3}  $. 
    
\end{proposition}
\begin{proof}
    We may assume that $g$ is sufficiently large, 
    as the statement holds for small values of $g$ by letting $c$ be sufficiently small.
    We set $p = \sqrt{\frac{150\log n}{n}}$, and we choose $n$ to be as large as possible so that $n$ is even and $pn^2 \leq g$.
    
    We let $G$ be an oriented graph on $n$ vertices
    which is randomly constructed as follows. For each pair of vertices $u$ and $v$, we join $u$ and $v$ by an edge $e$ independently with probability $p$, and if $e$ is added to $G$, we give $e$ one of the two possible orientations uniformly at random.
    By a Chernoff bound (see e.g.~\cite[Chapter 4]{Mitzenmacher}), it holds a.a.s.~that $|E(G)| < pn^2$, and hence $G$ a.a.s.~has Euler genus less than $g$.

    We aim to show that a.a.s., $\chi_2(G) > n/2$.
    To this end, we consider a 
    fixed
    coloring $\phi$ of $V(G)$ with $n/2$ colors, and we aim to estimate the probability that $\phi$ is a proper $2$-dipath coloring of $G$. 
    We obtain a subgraph $G'$ of $G$ by deleting at most one vertex from each color class of $\phi$ so that each color class of $G'$ has an even number of vertices. Clearly, $|V(G')| \geq n/2$. 
    We partition 
    each color class of $G'$ into vertex sets of size $2$, which gives a partition $\Pi$ of $V(G')$ in which each part $P \in \Pi$ consists of exactly two vertices which have the same color.
    We consider two distinct parts
    $P = \{u,v\}$ and $P' = \{u',v'\}$ in $\Pi$.
    We observe that if 
    $G'$ contains the arcs $u u'$ and $u'v$, then
    $\phi$ is not a $2$-dipath coloring of $G$. The probability that $G'$ contains both arcs $uu'$ and $u'v$ is $p^2 / 4$, and the number of ways to choose two 
    distinct parts $P,P' \in \Pi$
    is at least $\binom{\lceil n/4 \rceil }{2} > \frac{1}{36}n^2$. Therefore, the probability that $\phi$ is a
    proper $2$-dipath coloring of $G$ is at most 
    \[(1 - p^2 / 4)^{\frac{1}{36}n^2} < \exp \left ( - \frac{1}{144} (pn)^2 \right ).\]
    Therefore, by a union bound, the probability that $G$ has a $2$-dipath coloring is less than
    \[n^n \exp \left ( - \frac{1}{144} (pn)^2 \right ) = \exp \left ( n \log n - \frac{1}{144} (pn)^2 \right ) = o(1).\]
    Hence, $G$ a.a.s.~has no proper $2$-dipath coloring using $n/2$ colors. Therefore, a.a.s.,
    \[\chi_2(G) > n/2 = \frac{1}{2} \left ( \frac{pn^2}{\sqrt{150 \log n} }\right )^{2/3} =  \Omega\left ( \left ( \frac{g^2}{\log g} \right )^{1/3}  \right ).\]
    Finally, we may increase the Euler genus of $G$ to exactly $g$ without decreasing
    its $2$-dipath chromatic number
    by adding sufficiently many disjoint copies of $K_5$, completing the proof.
\end{proof}

\section{Further improving the oriented chromatic number for bound genus}
\label{sec:Improving_oriented}

Having established in Section~\ref{sec:oriented} that the oriented chromatic number of a graph $G$ is bounded by a polynomial function of its Euler genus $g$,
we turn our attention toward reducing the degree
of this polynomial. 
While we are unable to substantially improve the bound of $\chi_o(G) \leq g^{6400}$ given in Theorem~\ref{thm:g_oriented_intro},
we 
show that in order to improve the exponent of $6400$, it is sufficient to establish an improved upper bound for $\chi_2(G)$.
Unlike the oriented coloring problem,
which has global constraints,
the constraints of the $2$-dipath coloring problem are entirely local,
which often
makes $\chi_2(G)$ much easier to estimate than $\chi_o(G)$.
With this in mind, our main goal in this section is to prove
Theorem \ref{thm:g_improvement}, which shows
that an upper bound on $\chi_2(G)$ in terms of $g$ implies a similar upper bound on $\chi_o(G)$.

We borrow the following notation and definition from \cite{clow2023oriented}.
Given an oriented graph $G$, a vertex $v \in V(G)$, and an ordered vertex set 
$U = (u_1, \dots, u_t) \subseteq N(v)$, 
we write $F(U,v,G)$ for the vector in $\{-1,1\}^t$
whose $i$th entry is $1$ if $vu_i$ is an arc of $G$, and whose $i$th entry is $-1$ if $u_i v$ is an arc of $G$. 
Now, suppose $H$ is an oriented $k$-partite graph with exactly $N$ vertices in each partite set.
Let the partite sets of $H$ be called $P_1, \dots, P_k$.
We say that $H$ is \emph{$(k,d,N)$-full} if 
the following holds:
for each value $i \in  [k]$,
each
ordered subset $U = (u_1, \dots, u_d) \subseteq \bigcup_{j \neq i} P_j$ of size $d$, and each vector $q \in \{-1,1\}^d$,
there exists
a vertex $x \in P_i$
such that $F(U,x,H) = q$.

\begin{lemma}
\label{lem: (k,d,N)}
For each value $d \geq 2$ and $k \geq 5$, there exists a $(k,d, \lceil 8^d \log k \rceil )$-full graph.
\end{lemma}
\begin{proof}
    We let $N =  \lceil 8^d \log k \rceil $.
    We let $H$ be a random orientation of the complete $k$-partite graph $K_{N,\dots,N}$. We consider a fixed value $i \in \{1,\dots, k\}$ and a
    fixed ordered subset
    $U = (u_1, \dots, u_d) \subseteq \bigcup_{j \neq i} P_j$, as well as a fixed vector $q \in \{-1,1\}^d$.
    The probability that a given vertex $x \in P_i$ satisfies $F(U,x,G) = q$ is $2^{-d}$, so the probability that no vertex $v \in P_i$ satisfies $F(U,x,G) = q$ is at most $(1-2^{-d})^N < \exp(-2^{-d} N)$. Therefore, taking a union bound over all possible values $i \in [k]$, all ordered subsets $U \subseteq \bigcup_{j \neq i} P_j$ of size $d$, and all vectors $q \in \{-1,1\}^d$, the probability $p$ that $H$ is not $(k,d,N)$-full satisfies
    \[ p \leq k \cdot (kN)^d 2^d \exp(-2^{-d}N).\]
    The rest of the proof aims to show that $p < 1$. We observe that 
    \begin{eqnarray*}
    \log p &<& (d+1)(\log k + \log N + \log 2) - \frac{N}{2^d} \\
    &= & (d+1) (\log k + \log
    \lceil 8^d \log k \rceil + \log 2) - \frac{\lceil 8^d \log k \rceil }{2^d} \\ 
    &< & (d+1)(2 \log k + \log 8^d + \log 2) - 4^d \log k \\
    &=& (d+1) \left ( (2 - \frac{4^d}{d+1} ) \log k + (3d + 1) \log 2 \right )
    \end{eqnarray*}
The $\log k$ in the last expression has a negative coefficient for all $d \geq 2$, and therefore this expression is decreasing with respect to $k$.
 Hence, 
\[\log p < (d+1) \left ( (2 - \frac{4^d}{d+1} ) \log 5 + (3d + 1) \log 2 \right ),\]
which is negative for all $d \geq 2$.
Therefore, $p < 1$, and thus with positive probability, the oriented graph $H$ which we have constructed is $(k,d, \lceil 8^d \log k \rceil )$-full.
\end{proof}

Given two oriented graphs $G$ and $H$,
a function $h: V(G) \rightarrow V(H)$ is an oriented homomorphism if $(h(u),h(v)) \in A(H)$ for all $(u,v) \in A(G)$. That is, an oriented homomorphism is a connectivity-preserving and orientation-preserving map between the vertex sets of $G$ and $H$.
Sopena \cite{Sopena95} showed that $\chi_o(G) \leq k$ if and only if there exists an oriented homomorphism from $G$ to a graph on $k$ vertices.

The existence of $(k,d,N)$-full graphs is useful because their rich structure makes them a particularly well behaved target of an oriented homomorphism. In particular, Clow and Stacho \cite{clow2023oriented} show (in the proof of Theorem~2.1 in \cite{clow2023oriented}) that every $d$-degenerate graph satisfying $\chi_2(G) \leq k$ has an oriented homomorphism to every $(k,d,N)$-full graph. For completeness, we 
describe the same idea
again here in order to prove the following lemma.

\begin{lemma}
    \label{lem:XoX2}
    Let $d \geq 2$ and $k \geq 5$. If $G$ is a $d$-degenerate graph for which $\chi_2(G) = k$, then $\chi_o(G) \leq k \lceil 8^d \log k \rceil$.
\end{lemma}
\begin{proof}
    Consider some arbitrary orientation $\vec{G}$ of $E(G)$.
    We give $G$ a $d$-degeneracy ordering 
    $v_1,\dots,v_n$
    so that each vertex $v_i\in V(G)$ has at most $d$ neighbors $v_j$ for which $i > j$.
    We write $N = \lceil 8^d \log k \rceil$.
    We let $H$ be a $(k,d,N)$-full graph, which exists by Lemma \ref{lem: (k,d,N)}.
    We label the partite sets of $H$ as $P_1, \dots, P_k$.
    We will construct an oriented homomorphism $\phi:V(G) \rightarrow H$ 
    by choosing an image for the vertices of $G$ one at a time according to our degeneracy order of $V(G)$, which will show that $\chi_o(G) \leq |V(H)| = kN$.

    We construct $\phi$ as follows. 
    First, we fix a $2$-dipath coloring $\psi:V(G) \rightarrow \{1,\dots,k\}$ 
    of
    $\vec{G}$. 
    Then, we color the vertices of $G$ one at a time according to their degeneracy ordering $v_1, \dots, v_n$. 
    Each time we assign a color $\phi(v_{\ell})$ for a vertex $v_{\ell}$,
    we will require that $\phi(v_{\ell}) \in P_{\psi(v_{\ell})}$.
    Now, suppose we are considering a vertex $v_{\ell}$ and that a color $\phi(v_i) \in P_{\psi(v_i)}$ has already been assigned for each $1 \leq i < \ell$.
    We define $\phi(v_{\ell})$ as follows. Let $U = \{v_1, \dots, v_{\ell-1}\} \cap N(v_i)$. Fix some ordering $(u_1, \dots, u_t)$ of $U$, and let $q = F(U,v_{\ell},G)$.
    Given our assumption $\phi(v_i) \in P_{\psi(v_i)}$ 
    for each $1 \leq i < l$, and $\psi$ being a $2$-dipath coloring, if $u_i$ and $u_j$ have difference orientations to $v_{\ell}$, then $\psi(u_i) \neq \psi(u_j)$ implying $\phi(u_i) \neq \phi(u_j)$. Let $B = (\phi(u_1), \dots, \phi(u_t))$ be an ordered vertex subset of $V(H) \setminus P_{\psi(v_{\ell})}$ without repeated vertices. As we supposed that $\phi(v_i) \in P_{\psi(v_i)}$, it follows that that $B \subseteq V(H) \setminus P_{\psi(v_{\ell})}$. Hence, since $H$ is $(k,d,N)$-full, there exists a vertex
    $x \in P_{\psi(v_{\ell})}$ such that for all $i,j$ if $u_iv_{\ell},v_{\ell}u_j \in A(G)$, then $\phi(u_i)x,x\phi(u_j) \in A(H)$. We map $\phi(v) = x$. It is straightforward to check that $\phi$ is an oriented homomorphism, and hence the proof is complete.
\end{proof}

Now, we are ready to prove Theorem \ref{thm:g_improvement}.
We prove the following stronger theorem.

\begin{theorem}
    There exists a constant $C$ such that if $G$ is a graph of Euler genus $g \geq 0$ satisfying $\chi_2(G) = k$, then 
    \[\chi_o(G) <  C (k \log k + g + 1).\]
\end{theorem}
\begin{proof}
    We will show that the constant $C = 2^{20}$ is sufficiently large.
    We write $n = |V(G)|$. If $n \leq 6g$, then $\chi_o(G) \leq 6g$.
    If $g \leq 1$, then we may write $a = \chi_a(G)$ and use the inequalities $a \leq 7$ \cite{AMS} and $\chi_o(G) \leq a 2^{a - 1} \leq 448$ \cite{raspaud1994good} to finish the proof.

    Otherwise, we assume that $g \geq 2$ and $n > 6g$.
    We order the vertices of $G$ as $v_1, \dots, v_n$ as in the proof of Theorem \ref{thm:g_intro} and Theorem \ref{thm:g_oriented_intro}, so that for each $i \in [n]$, $v_i$ is has minimum degree in the graph $G[\{v_1, \dots, v_i\}]$. 
    We write $V_1 = \{v_1, \dots, v_{6g}\}$ and $V_2 = \{v_{6g+1},\dots,v_n\}$.
    We define $G' = G \setminus E(G[V_1])$, and as in the proof of Theorem \ref{thm:g_oriented_intro}, $G'$ is a $6$-degenerate graph.

    If $\chi_2(G') < 5$, then 
     $\chi_o(G') \leq 8$
    by the inequality $\chi_o(G') \leq 2^{\chi_2(G') - 1}$ \cite{macgillivray2010injective}.
    Otherwise,
    by Lemma \ref{lem:XoX2}, $\chi_o(G') < \chi_2(G')\lceil 8^6 \log \chi_2(G')\rceil \leq k\lceil 8^6 \log k \rceil < C k \log k$.
    In both cases, by Lemma \ref{lem:minusE}, \[\chi_o(G) \leq |V_1| + \chi_o(G') < C(k\log k + g +1).\]
    This concludes the proof.
    \end{proof}

\section{Conclusion}


While our results give asymptotic bounds for the values $\chinj(G)$ and $\chi_o(G)$ of graphs $G$ with large Euler genus $g$, 
the optimal upper bounds for these parameters when $g$ is small are still unknown. For small values of $g$, estimates for $\chinj(G)$ and $\chi_o(G)$
rely on bounds in terms of acyclic chromatic number, namely $\chinj(G) \leq 3 \binom{\chi_a(G)}{2}$ and $\chi_o(G) \leq \chi_a(G) 2^{\chi_a(G) - 1}$.
For planar graphs $G$ in particular,  these upper bounds tell us that $\chinj(G) \leq 3 \binom{5}{2} = 30$, and $\chi_o(G) \leq 5 \cdot 2^{5-1} = 80$. We do not know whether these bounds are even close to being tight, 
and it is only known that there exist planar graphs $G$ satisfying $\chinj(G) \geq 18$ and $\chi_o(G) \geq 18$.
While Axenovich et al.~\cite{Axenovich} proved that the inequality $\chinj(G) \leq 3 \binom{\chi_a(G)}{2}$ is tight,
it is unknown if the inequality is tight for planar graphs. 
Even more confounding, 
it is unknown if the inequality 
$\chi_o(G) \leq \chi_a(G) 2^{\chi_a(G) - 1}$ is ever tight when $\chi_a(G) \geq 2$,
as Kostochka, Sopena, and Zhu \cite{kostochka1997acyclic} 
only give examples of graphs $G$ satisfying $\chi_o(G) \geq 2^{\chi_{a}(G) - 1} - 1$.
Therefore, we would like to emphasize the following questions:

\begin{question}
\label{q:ChinjXa}
    Does $\chinj(G) < 3\binom{\chi_a(G)}{2}$ hold for every planar graph $G$ with at least one edge? 
\end{question}

\begin{question}
\label{q:XoXa}
    Does $\chi_o(G) < \chi_a(G) 2^{\chi_a(G) - 1}$ hold for every graph $G$ with at least one edge?
\end{question}

For graphs $G$ of large Euler genus $g$,
there is a large gap between our general upper bound $\chi_o(G) \leq g^{6400}$ and
the lower bound $\chi_o(G) \geq g^{\frac{2}{3} - o(1)}$ achieved by our construction in Proposition \ref{prop: oriented_lower}.
Therefore,
it is natural to ask for the 
optimal exponent in the upper bound of 
$\chi_o(G)$:

\begin{question}
\label{conj: Linear_oriented_genus}
    What is the least $k$ such that $\chi_o(G) = O(g^k)$ for every graph $G$ of Euler genus $g$? 
\end{question}

One natural possibility is that $k = \frac{2}{3}$ is best possible.
Indeed, the random construction in Proposition \ref{prop: oriented_lower}
gives a graph of Euler genus $g$ with oriented chromatic number $\chi_o(G) = \Omega \left ( (g^2 / \log g)^{1/3} \right )$,
and similar random constructions give bounds for other coloring parameters, such as acyclic chromatic number \cite{AMR,AMS}, star chromatic number \cite{FRR}, and 
nonrepetitive chromatic number \cite{AlonNR}, which are tight up to some logarithmic factor.
We also observe that Theorem~\ref{thm:g_improvement} implies that if $\chi_2(G) = O(\frac{g}{\log g})$, then $k \leq 1$. To that end,
it is natural to ask if the $2$-dipath chromatic number is sublinear in the genus of a graph.

\begin{question}
\label{conj: SubLinear_2-dipath_genus}
    For every graph $G$ of  Euler genus $g$, is $\chi_2(G) = o(g)$?
\end{question}

Finally, one can ask if the upper bound for the oriented chromatic number in terms of the $2$-dipath chromatic number of a $d$-degenerate graph used in Lemma \ref{lem:XoX2} is best possible. Of particular interest to the authors is the following conjectured improvement.

\begin{conjecture}
 For every integer $d \geq 1$, there exists a constant $c$ depending on $d$ such that
 $\chi_o(G)\leq c\chi_2(G)$ for every $G$ with degeneracy at most $d$.
\end{conjecture}

This would improve the bound given in Lemma~\ref{lem: (k,d,N)} by a logarithmic factor as well as improving bounds from \cite{clow2023oriented} and \cite{macgillivray2010injective} by an even larger margin. Such an improvement could be achieved by constructing small $(k,d,N)$-full graphs, or finding other graphs with similar properties. Such a improvement would remove a logarithmic factor from the upper bound in Theorem~\ref{thm:g_improvement} in the case that there exists graphs with $\chi_2(G) = \Omega(g)$.

\bibliographystyle{plain}
\bibliography{injbib}

\end{document}